\documentclass[a4paper,12pt]{amsart}

\usepackage{amsfonts,bm}
\usepackage{amssymb}
\usepackage{setspace}
\usepackage{fancyhdr}
\usepackage{color}
\usepackage{rotating}
\usepackage{mathdots}
\usepackage[tableposition=top]{caption}
\usepackage{picinpar,graphicx}
\usepackage{tabularx}
\newcolumntype{Y}{>{\small\raggedright\arraybackslash}X}

\usepackage{amscd}
\usepackage{indentfirst}
\usepackage{latexsym}
\usepackage{amsthm}
\usepackage{amsmath}
\usepackage{tabularx}
\newtheorem{theorem}{Theorem}[section]

\newtheorem{lemma}{Lemma}[section]
\newtheorem{corollary}{Corollary}[section]
\newtheorem{remark}{Remark}[section]
\newtheorem{fact}{Fact}[section]

\newcommand{\mcp}{\mathcal{P}}
\newcommand{\mcq}{\mathcal{Q}}

\newcommand{\G}{\mathcal{G}_0}

\DeclareMathOperator{\ls}{LS}
\DeclareMathOperator{\lig}{LI}
\DeclareMathOperator{\Lig}{L}

\DeclareMathOperator{\card}{card}
\DeclareMathOperator{\sh}{Sh}

\DeclareMathOperator{\dx}{dx}

\DeclareMathOperator{\C}{C}

\begin{document}
\title[Possible $g$-entropy convergence rates]{Possible generalized entropy convergence rates}
\author{Fryderyk Falniowski}

\address[F.~Falniowski]{Department of Mathematics, Cracow University of Economics,
Rakowicka~27, 31-510 Krak\'ow, Poland}
\email{fryderyk.falniowski@uek.krakow.pl}
\date{\today}

%\end{center}

%\setcounter{theo}{0}
%\setcounter{cor}{0}

\begin{abstract} 
We consider an isomorphism invariant for measure-preserving systems-- types of generalized entropy convergence rates. We show the connections of this invariant with the types of Shannon entropy convergence rates. In the case when they differ we show several facts for aperiodic, completely ergodic and rank one systems.  We use this concept to distinguish some measure-preserving systems with zero entropy.
\end{abstract}

\maketitle

\section{Introduction}
The question whether two measure-preservng dynamical systems are isomorphic is one of the fundamental questions of the theory of dynamical systems. Therefore, the search for isomorphism invariants, which would establish that two systems are non-isomorphic is of the main interest. One of the most useful isomorphism invariants is Kolmogorov-Sinai entropy. Unfotunately, in general, systems with equal entropy may be non-isomorphic. In particular, in the case of the huge class of zero entropy systems we need to use other tools.

Zero entropy systems are much less complex than those with positive entropy, but their complexity and dynamics may vary considerably. %Of course, by the Variational Principle for every $T$-invariant measure the Kolmogorov-Sinai entropy of $T$ is equal to zero.  
Such systems have been studied for many years. %Lately, Serafin \cite{serafin} showed that there is no universal topological dynamical system in the class of all automorphisms with zero entropy. 
At the turn of the century  many tools for distinguish them were introduced. It is worth to mention the concept of measure-theoretic \cite{FerencziMTC}, topological \cite{Blanchard} metric \cite{GalatoloMTC} and symbolic \cite{Ferenczicomplexity} complexity,  generalized topological entropy \cite{ChengLi, Galatolo}, measure-theoretic and topological entropy dimension \cite{MdCarvalho, ChengLi, DouHuangPark, FerencziPark}, a~slow entropy type invariant proposed by Katok and Thouvenot \cite{Hochman, KatokThouvenot}, and types of entropy convergence rates \cite{Blume98,Blume2000, Blumenotpubl, Blume97, Blume11.1, Blume11.2}.

Dynamical and Kolmogorov-Sinai entropies are basic tools for investigating dynamical systems.  % It appeared to be an exceptionally powerful tool for exploring nonlinear systems.
%One of the biggest advantages of the Kolmogorov-Sinai entropies lies in the fact that it makes possible to distinguish the formally regular systems (those with the measure-theoretic entropy equal to zero) from the chaotic ones (with positive entropy, which implies positivity of topological entropy \cite{Misiurewicz}).
 %The Kolmogorov-Sinai entropy of a~given transformation $T$ acting on a~probability space $(X,\Sigma,\mu)$ is defined  as the supremum over all finite measurable partitions $\mathcal{P}$   of the dynamical entropy of $T$ with respect to $\mathcal{P}$, denoted by $h(T,\mathcal{P})$.
As a~dynamical counterpart of Shannon entropy, the entropy of transformation $T$ with respect to a~given partition $\mathcal{P}$ (called also the dynamical entropy of $\mcp$) is defined  as the limit of the sequence $\left(\frac 1n H(\mathcal{P}_n)\right)$, where
\[
H(\mathcal{P}_n)=\sum\limits_{A\in \mathcal{P}_n} \eta\left(\mu(A)\right),
\]
with $\eta$ being the Shannon function given by $\eta(x)=-x\log x$ for $x>0$ with $\eta(0):=0$
and $\mathcal{P}_n$ is the join partition of partitions $T^{-i}\mathcal{P}$ for $i=0,...,n-1$.
%The existence of the limit in the definition of the dynamical entropy follows from the subadditivity of $\eta$. 
The most common interpretation of this quantity is the average (over time and the phase space) one-step gain of information about the initial state. %Taking supremum over all finite partitions we obtain the Kolmogorov-Sinai entropy. 

 In 1997, Frank Blume proposed \cite{Blume97} to analyze zero entropy  dynamical systems by observing  convergence rate of $ H(\mcp_n)/n $ to the dynamical entropy of $\mcp$ for certain classes of finite partitions of $ X $. In subsequent works \cite{Blume98, Blume2000, Blumenotpubl, Blume11.1, Blume11.2} he obtained several results characterizing the rate of convergence for completely ergodic and rank one systems.
In particular, he showed how this concept might be used to distinguish non-isomorphic rank one systems. 

He suggested  (in \cite{Blume97}) that one searching for new isomorphism invariants should use entropy functions different than $\eta$. The analysis of convergence rates of partial $g$-entropies $(H(g,\mcp_n)/n)$, where
\[H(g,\mcp_n)=\sum_{A\in\mcp_n}g(\mu(A)),\]
to the limit (for $g$~which behaves  differently than  $\eta$~in the neighbourhood of zero), gives a~chance to capture the differences in behaviour, which we are not able to observe  analyzing the convergence rates of the Shannon partial entropies. In this note, following Blume's suggestion, we generalize his proposal to the case of any entropy function. From  \cite[Thm 3.4, Cor. 3.5]{FalniowskiConnections} it follows that the Kolmogorov-Sinai entropy type invariant being the supremum over all finite partititions of the dynamical  $g$-entropy will not, in general help us to differ systems of equal metric entropy.
We will show that the analysis of the $g$-entropy convergence rates  allows us to obtain   an isomorphism invariant, so called types of $g$-entropy convergence, which may be useful in some nontrivial cases when the entropy fails. For simplicity we assume that the limit $\lim\limits_{x\to 0^+}g(x)/\eta(x)$ exists, therefore, considered functions belong to $\G^0$, $\G^{\sh}$ or $\G^{\infty}$ and the limit, which we call the dynamical $g$-entropy of $\mcp$, \[h_{\mu}(g,T,\mcp)=\lim_{n\to\infty}\frac 1n H(g,\mcp_n)\] exists.

We begin by introducing the relevant concepts and  basic facts. Then we compare types of $g$-entropy and $\eta$-entropy convergence rates. Then, we show how this invariant can be used e.g. for completely ergodic systems. Searching for a~new invariant isomorphism we will thoroughly discuss the extreme case  -- $g\in\mathcal{G}_0^0$. Finally, we construct a~class of weakly mixing, rank one systems in which types of $g$-entropy convergence rates are useful.

\section{Basic facts and definitions}
Let $(X,\Sigma,\mu)$ be a Lebesgue space and let $g:[0,1] \mapsto \mathbb{R}$ be a~concave function with $g(0)=\lim\limits_{x\to 0^+}g(x)=0$. By $\mathcal{G}_0$ we will denote the set of all such functions, and each $g\in\G$ will be called an entropy function. Every $g\in\G$ is subadditive, i.e.  $g(x+y)\leq g(x)+g(y)$ for every $x,y\in[0,1]$, and quasihomogenic, i.e.  $\varphi_g\colon (0,1]\to\mathbb{R}$ defined by $\varphi_g(x):=g(x)/x$ is decreasing (see \cite{Rosenbaum}).\footnote{If $g$ is fixed we will omit the index, writing just $\varphi$.} Any finite family of pairwise disjoint subsets of $X$ such that $\bigcup\limits_{A_i\in\mcp}A_i=X$ is called a~\textsf{partition}.  The set of all finite partitions of $X$~will be denoted by $\mathfrak{B}$ and by $P(X)$ we will denote the set of all nontrivial binary partitions of $X$: \[P(X):=\left\{\mcp\;|\;\mcp=\{E,X\backslash E\}\;\textrm{for some}\; E\in \Sigma,\;\textrm{ such that }\; 0<\mu(E)<1\right\}.\]  We say that a~partition $\mcq$ is {\sf a~refinement of} $\mcp$ (and write $\mcp \preccurlyeq \mcq$) if every set from $\mcp$ is an algebraic sum of sets from $\mcq$.  The \textsf{join partition} of $\mathcal{P}$ and $\mathcal{Q}$ (denoted by $\mathcal{P}\vee\mathcal{Q}$) is a~partition, which  consists of the subsets of the form $B\cap C$ where $B\in\mathcal{P}$ and
$C\in\mathcal{Q}$. The join partition of more than two partitions is defined similarly.

For a~given finite partition $\mathcal{P}$ %\footnote{If $\mcp$ consists of two elements $\{E,X\backslash E\}$, then we will denote this partition by $\mcp^E$.}  
%and any measurable set $A\in\Sigma$,  $\mathcal{P}/A:=\{E_1\cap A,...,E_k \cap A\}$ is a finite partition of $A$ induced by $\mathcal{P}$.
 we define the \textsf{$g$-entropy of the partition $\mathcal{P}$} as
\begin{equation}
H(g,\mathcal{P}):=\sum_{A\in\mcp} g\left(\mu(A)\right).
\end{equation}

If $g=\eta$ the latter is equal to the Shannon entropy of the partition $\mathcal{P}$.
 For an automorphism $T\colon X\mapsto X$ and a~partition $\mathcal{P}=\{E_{1},...,E_{k}\}$ we put
\[
T^{-j}\mathcal{P}:=\{T^{-j}E_{1},...,T^{-j}E_{k}\}
\]
and
\[
\mathcal{P}_{n}=\mathcal{P}\vee T^{-1}\mathcal{P}\vee...\vee T^{-n+1}%
\mathcal{P}.%
\]

Now for a~given $g\in\G$ and a finite partition $\mcp$ we can define the \textsf{dynamical $g$-entropy of the transformation $T$ with respect to $\mathcal{P}$}  as

\begin{equation}
\label{dynentr} h_{\mu}(g,T,\mcp)=\limsup_{n\to\infty}\frac 1n H\left(g,\mcp_n\right).
\end{equation}
%Alternatively we will call it the $g$-entropy of the process $(X,\Sigma,\mu,T,\mcp)$. 
If the dynamical system $(X,\Sigma,T,\mu)$ is fixed then we omit $T$,  writing just $h(g,\mcp)$.  As in the case of Shannon dynamical entropies we are interested in the existence of the limit of $\left( H(g,\mcp_n)/n\right)$. If $g=\eta$, we obtain the Shannon dynamical entropy $h(T,\mcp)$. However, in the general case we can not replace an upper limit  in (\ref{dynentr}) by the limit, since it might not exist (see \cite{FalniowskiConnections}). Existence of the limit in the case of the  Shannon function follows from the subadditivity of the static Shannon entropy. This property has every {\sf subderivative function}, i.e. a~function for which the inequality $g(xy)\leq xg(y)+yg(x)$ holds for any $x,y\in[0,1]$ (the subclass of $\G$ of functions, which fulfill this condition will be denoted by $\G'$), but this is not true in general %(an appropriate example will be given in Section 2.2). Therefore we consider more general classes of functions for which the limit exists.
 It  exists \cite{FalniowskiConnections}, if $g$~belongs to one of the following classes:
\[
\mathcal{G}_{0}^{0}:=\left\{g\in \mathcal{G}_0\;\left|\; \lim_{x\to 0^+} \frac{g(x)}{\eta(x)}=0\right.\right\}
\;\;\;\text{or}\;\;\;
\mathcal{G}_{0}^{\sh}:=\left\{g\in \mathcal{G}_0\;\left|\; 0<\lim_{x\to 0^+} \frac{g(x)}{\eta(x)}<\infty \right.\right\}.
\]
It is easy to show that if $g$ is subderivative then the limit $\lim\limits_{x\to 0^+}g(x)/\eta(x)$ is finite, so $\G'\subset\G^0\cup\G^{\sh}$. We will also consider functions from the following class

\[\mathcal{G}_{0}^{\infty}:=\left\{g\in \mathcal{G}_0\;\left|\; \lim_{x\to 0^+} \frac{g(x)}{\eta(x)}=\infty \right.\right\}.\]

In this note we will consider functions from $\G$, for which the limit $\lim\limits_{x\to 0^+}g(x)/\eta(x)$ exists, so $g$~will belong to one of three classes defined above.

We say that $(X,\Sigma,\mu,T)$ is  {\sf aperiodic}, if  \[\mu\left(\{x\in X: \exists n\in\mathbb{N}\; \;T^nx=x\}\right)=0.\] 

If $M_0,\ldots,M_{n-1}\subset X$ are pairwise disjoint sets of equal measure, then $\tau=(M_0,M_1,\ldots,M_{n-1})$ is called  a~{\sf tower}. 
If additionally $M_k=T^{k} M_{0}$ for $k=1,\ldots,n-1$, then $\tau$ is called a~{\sf Rokhlin tower}.\footnote{It is also known as Rokhlin-Halmos or Rokhlin-Kakutani tower.}
By the same bold letter $\bm{\tau}$ we will denote the set $\bigcup\limits_{k=0}^{n-1}M_k$. Obviously $\mu(\bm{\tau})=n\mu(M_{n-1})$. Integer $n$ is called {\sf the height } of tower $\tau$. Moreover for $i<j$ we define a subtower \[\tau_i^j:=\left(M_i,\ldots,M_j\right) \;\;\text{and}\;\;\bm{\tau}_i^j=\bigcup_{k=i}^j M_k. \] 
By the Rokhlin Lemma in aperiodic systems there exist Rokhlin towers of a~given length, covering sufficiently large part of $X$. 

Our goal is to find a lower bound for the dynamical $g$-entropy of a given partition.
For this purpose we will use Rokhlin towers and we will calculate dynamical $g$-entropy with respect to a~given Rokhlin tower. This leads us to the following definition:
Let $\mcp$ be a finite partition of $X$ and  $F\in\Sigma$, then we define the (static) {\sf $g$-entropy of $\mcp$ restricted to $F$} as \[H_F (g,\mcp):=\sum_{B\in\mcp}g(\mu(B\cap F)).\]

The following lemma \cite[Lemma 2.11]{FalniowskiConnections} gives us the estimation for $H(g,\mcp)$ from below by the value of $g$-entropy restricted to a~subset of $X$. 

\begin{lemma} \label{lemmaH-Hzaw}
Let $g\in\mathcal{G}_0$. Let $\mcp\in\mathfrak{B}$ be such that there exists a~set $E\in\mcp$ with $0<\mu(E)<1$. If $F\in\Sigma$, then
 \[H(g,\mcp)\geq H_{F}(g,\mcp)-\left|g_-'\left( 1/2\right)\right|-d_{\max},\]
where $d_{\max}:=\max\limits_{x,y\in[0,1]}|g(x)-g(y)|$.
\end{lemma}
Definition of the $g$-entropy of a~partition and concavity of $g$~implies that
\begin{lemma}
\label{lemlambda} Let  $\mcp\in\mathfrak{B}$, $F\in\Sigma$ and $\lambda>0$ be such that $\mu(B\cap F)\leq \lambda$ for every $B\in\mcp$. Then \[H_F(g,\mcp)\geq \varphi(\lambda)\mu(F).\]
\end{lemma}

\begin{lemma}\label{gxiogr}
If $g\in\G'$ and $x_i\geq 0$ ($i=1,\ldots,n$) are such that $\sum\limits_{i=1}^nx_i\leq 1$, then \[\sum_{i=1}^ng(x_i)\leq \max_{x\in[0,1]}g(x)+ng\left(\frac 1n\right)\sum_{i=1}^{n}x_i\]
\end{lemma}
\begin{proof}
It follows from the fact that \[\sum_{i=1}^n g(x_i)\leq n g\left(\frac 1n \sum_{i=1}^nx_i\right)\leq g\left(\sum_{i=1}^nx_i\right)+ng\left(\frac 1n\right)\sum_{i=1}^n x_i,\]
which completes the proof of lemma.
\end{proof}

We say that $(X,\Sigma,\mu,T)$ is ergodic if for every $A\in\Sigma$ with $T^{-1}A=A$ either $\mu(A)=0$ or $\mu(A)=1$ and is completely ergodic if and only if $T^n$ is ergodic for every positive integer $n$.
\subsection{Generalized entropy convergence rates}
Let $g\in\G$. We introduce types of $g$-entropy convergence rates. Let $(X,\Sigma,\mu,T)$ be a~measure-preserving system, $(a_n)$ a~(strictly) increasing sequence with $\lim\limits_{n\to\infty}a_n=\infty$ and $c\in [0,\infty)$. If $P$ is a~class of finite partitions of $X$, then we say that $(X,T)$ is {\sf of type $(\ls(g)\geq c)$ for  $((a_n),P)$} if \[\limsup_{n\to\infty}\frac{H(g,\mcp_n)}{a_n}\geq c \;\; \textrm{for all}\; \mcp\in P,\]
and $(X,T)$ is {\sf of type $(\lig(g)\geq c)$ for  $((a_n),P)$} if
 \[\liminf_{n\to\infty}\frac{H(g,\mcp_n)}{a_n}\geq c \;\; \textrm{for all}\; \mcp\in P.\]
Analogously we define types $(\ls(g)\leq c)$, $(\lig(g)\leq c)$ and $(\ls(g)<\infty)$, $(\ls(g)=\infty)$, $(\lig(g)< \infty)$, $(\lig(g)=\infty)$, $(\ls(g)>0)$ and $(\lig(g)>0)$.
From now on, we always assume that the limit  $\lim\limits_{x\to 0^+}g(x)/\eta(x)$ exists, so $g\in\G^0\cup\G^{\sh}\cup\G^{\infty}$. 

 If for a~given  $(a_n)$ and any $\mcp\in P$ there exists the limit  (finite or not) $\lim\limits_{n\to\infty}H(g,\mcp_n)/a_n$, then $(X,T)$ is of sufficient type $\Lig (g)$. The $g$-entropy convergence rate was introduced by Blume \cite{Blume97} (however no strict results were given) and is a~natural generalization of entropy convergence rates, since we obtain them for $g=\eta$.

If $(X,T)$  is a~measure-preserving system and $g\in\mathcal{G}_0^{0}\cup \G^{\sh}$, then the limit  $\lim\limits_{n\to\infty} H(g,\mcp_n)/n$ exists and is finite for every finite partition of $X$ \cite[Cor. 2.7.3]{FalniowskiConnections}. Therefore every measure-preserving system is of type $(\ls(g)<\infty)$ (and of type $(\lig(g)< \infty)$) for $((n),P)$.

Through the main part of this text we will consider  the class $P(X)$ and we will concentrate our attention on the choice of $(a_n)$ and $g$. The reason for choosing $P(X)$ as our standard class is twofold: on the one hand $P (X)$ is simple enough to reduce the complexity of many proofs, and on the other, it is large enough for generalized entropy convergence types to become isomorphism invariants.
More precisely if  $(X,T)$ and $(Y,S)$ are isomorphic, then $(X,T)$ is of type  $(\ls(g)\geq c)$ for $((a_n),P(X))$, if and only if, $(Y,S)$ is of type $(\ls(g)\geq c)$ for $((a_n),P(Y))$.
It follows from the observation that if  $\psi:X\mapsto Y$ is an isomorphism, then \[H(g,\mcp_n)=H(g,\psi\mcp_n)\;\; \textrm{for}\; \mcp\in P(X)\;\;\; (n\in\mathbb{N}),\] where $\psi\mcp:=\{\psi(A)\;|\;A\in\mcp\}$. %Zatem jeśli $(X,T)$ jest typu $(\ls(g)\geq c)$ dla $((a_n),P(X))$, to dla dowolnego $\mcp \in P(Y)$ mamy \[\limsup_{n\to\infty}\frac{H(g, \mcp_n)}{a_n}=\limsup_{n\to\infty}\frac{H(g, \phi^{-1} \mcp_n)}{a_n}\geq c.\]
Similar statements are obviously true for all the others convergence types. Therefore we may use $g$-entropy convergence types to investigate dynamical systems with the same $g$-entropy (and standard dynamical entropy).

In the case when $(X, T)$ has zero entropy and $g\in \G^0\cup\G^{\sh}$, we have $\lim\limits_{n\to\infty} H(g,\mcp_n)/n=0$ for $\mcp\in\mathfrak{B}$. Therefore, it is reasonable to consider for systems with zero entropy only such $(a_n)$, for which $\lim\limits_{n\to\infty}a_n=\infty$ and $\lim\limits_{n\to\infty}a_n/n=0$. We will call each such sequence {\sf a~sequence with sublinear growth}.

\subsubsection{Symbolic representation of atoms}

We will introduce a~notation, which we will use throughout the rest of our discussion.
In order to show that $ (X,T)$ is of a~certain convergence type, we need to find estimates for $H(g,\mcp_n)$, $\mcp\in P(X)$ and this requires
finding ways to analyze join partitions $\mcp_n$. Our most important tool will be a~symbolic representation of atoms in $\mcp_n$ by their 01-names. For a~given  $E\in\Sigma$,  $x\in X$,  $n\in\mathbb{N}$ and  $\mcp=\mcp^E=\{E,X\backslash E\}$ we set
\[s_n^E(x):=\left(\bm{1}_E(x),\bm{1}_E(Tx),\ldots,\bm{1}_E(T^{n-1}x)\right),\]
and if $s\in\{0,1\}^n$, we define
\begin{equation} \label{AnEs}A_n^E(s):=(s_n^E)^{-1}(\{s\}).\end{equation}
Then we have \begin{equation} \label{PnAnE}\mcp_n=\{A_n^E(s_n^E(x))\;|\;x\in X\}=\{A_n^E(s)\;|\;s\in\{0,1\}^n\}.\end{equation}
For a given word  $s=(s_0,...,s_{n-1})\in\{0,1\}^n$ {\sf the period of $s$}~is \[p_n(s):=\min\{k\in\{0,\ldots,n-1\} \;|\;s_i=s_{i+k}\;\text{for all} \; i\in\{0,\ldots,n-k-1\}\}\cup\{n\}. \]

 This symbolic representation will be combined with the use of
Rokhlin towers $\tau=(M, TM, . . . , T^{n-1}M)$ which will cover a sufficiently large portion
of $ X$. In particular, we will be interested only in the entropy of $\mcp_n$ with respect
to the set ${\bm \tau}=\bigcup\limits_{k=0}^{n-1}T^i M$.
We will use the following lemma \cite[Lemma 1.6]{Blume97} :
\begin{lemma}\label{lemma16Blume}
If $E\in\Sigma$ and $s\in\{0,1\}^{2^n}$, then $\mu(A_n^E(s))\leq 1/p_n(s)$.
\end{lemma}

\section{Connections between  ($\ls(g)\geq c/\lig(g)\geq c$)  and  \mbox{($\ls\geq c/\lig\geq c$)} convergence types }
We analyze the connections between $g$-entropy and $\eta$-entropy convergence rates. Since the considerations for lower and upper limits are similar we discuss only the case of  $\ls$ types. Let $c>0$.

\subsection{ $g\in\mathcal{G}_0^{\sh}$}
 In this case \cite[Corollary 2.7.3]{FalniowskiConnections} implies that $\lim\limits_{n\to\infty}H(g,\mcp_n)/H(\mcp_n)$ is positive. The equality 
\[\frac{H(g,\mcp_n)}{a_n}=\frac{H(g,\mcp_n)}{H(\mcp_n)}\cdot\frac{H(\mcp_n)}{a_n}\;\;\;\;(n\in\mathbb{N}),\]
guarantees that for a~given pair $\left((a_n),P(X)\right)$  a system $(X,T)$ is
\begin{center}
of type ($\ls\leq c$) if and only if it is of type $(\ls(g)\leq \C(g)\cdot c)$,\\
of type ($\ls\geq c$) if and only if it is of type ($\ls(g)\geq \C(g)\cdot c$),
\end{center}
and
\begin{center}
of type ($\ls<\infty$) if and only if it is of type $(\ls(g)<\infty)$,\\
of type ($\ls=\infty$) if and only if it is of type ($\ls(g)=\infty$).
\end{center}
  
Therefore, if $g$~behaves like $\eta$ in the neighbourhood of zero, then  studying $\ls(g)$ convergence types will not attain any new information about the system (in addition to the attained from the Shannon entropy convergence types).
We obtain natural generalizations of theorems for the Shannon convergence entropy rates   \cite[Thm 2.8, 3.9, Cor. 3.10, 3.11]{Blume97}:

\begin{corollary}\label{genofthm2.8}
Let $g\in\mathcal{G}_0^{\sh}$. If $(X,T)$ is aperiodic and $(a_n)$ is a~sequence with sublinear growth, then there exists a~partition $\mcp^E=\{E,X\backslash E\}\in\mathfrak{B}$, such that \[\lim_{n\to\infty}\frac{H(g,\mcp_n^E)}{a_n}=\infty.\]
In other words $(X,T)$ is not of type ($\lig(g)<\infty$) for $((a_n),P(X))$.
\end{corollary}

\begin{corollary}\label{genofthm3.9}
Let $g$ and $(X,T)$ be as in Corollary \ref{genofthm2.8} and $\phi\colon [0,\infty) \mapsto (0,\infty)$ is an increasing function with \[\int\limits_1^{\infty}\frac{\phi(x)}{x^2} \dx<\infty.\] 
Let \begin{equation} \label{Rx}R(X):=\left\{\mcp\in P(X),\;\text{such that} \; \lim_{n\to\infty}\max \{\mu(A)|A\in \mcp_n\}=0\right\}\end{equation}
Then for every $\mcp\in R(X)$ 
\[\limsup_{n\to\infty}\frac{H(g,\mcp_n)}{\phi(\log_2 n)}=\infty.\]
Equivalently,  $(X,T)$ is of type ($\ls(g)=\infty$) for $((\phi(\log_2 n)),R(X))$.
\end{corollary}

\begin{corollary}  \label{wn341tempa} Let $(X,T)$ be ergodic and $\phi$ given as in Corollary \ref{genofthm3.9}. If $E\in \mathcal{B}$ is such that  $\mcp_{n-1}^E\neq \mcp_n^E$ for every  $n\in\mathbb{N}$, then \[\limsup_{n\to\infty}\frac{H(g,\mcp_n^E)}{\phi(\log_2 n)}=\infty.\]
\end{corollary}

\begin{corollary} \label{wn342tempa} Let $(X,T)$ be completely ergodic and $\phi$ be as in Corollary \ref{genofthm3.9}. If $\mcp\in P(X)$, then \[\limsup_{n\to\infty}\frac{H(g,\mcp_n)}{\phi(\log_2 n)}=\infty.\]
\end{corollary}

\subsection{$g\in\G^0$}
In this case the diversity is greater. Of course, each system of type $(\ls\leq c)$ for some $((a_n),P(X))$ is of type $(\ls(g)=0)$ for $((a_n),P(X))$ and hence of type $(\Lig(g)=0)$. 
Therefore we concentrate our attention on the case when the system is of type  $(\ls \geq c)$ for $((a_n),P(X))$.  
It is also easy to see that if $g\in\G^0$ and the system is of type $(\ls (g)\geq 0)$, then there exists a~partition  $\mcp\in P(X)$ such that the upper limit of $H(\mcp_n)/a_n$ is infinity and if   $(X,T)$ is of type $(\ls(g)>0)$ for $((a_n),P(X))$, then it is of type $(\ls =\infty)$ for the considered pair. 
To understand how $g$-entropy convergence rates differ from the classical entropy convergence rates, consider the following example:
Let $(X,\sigma)$ be a~subshift of the fullshift over an alphabet $\mathcal{A}=\{0,1\}$ and $g$~is given by (\ref{blumefunction0}). Then for every $\mcp\in P(X)$ we have \[\limsup_{n\to\infty}\frac{H(g,\mcp_n)}{\log n}\leq \limsup_{n\to\infty}\frac{\varphi(1/\card \mcp_n)}{\log n}\leq \limsup_{n\to\infty}\frac{\log (1+n)}{\log n}=1.\] Therefore every subshift over two symbols is of type $(\ls (g)\leq 1)$ for $\left(\left(\log n\right),P(X)\right)$ and of type $(\ls (g)=0)$ for $\left(\left(a_n\right),P(X)\right)$, for any  $(a_n)$, for which $\lim\limits_{n\to\infty}(\log n)/a_n=0$. 
Moreover for every subshift, $g\in\mathcal{G}_0^{0}$ and $\mcp\in P(X)$, we have \[\frac{H(g,\mcp_n)}{a_n}\leq \frac{\varphi(1/\card \mcp_n)}{a_n}\leq \frac{\varphi(2^{-n})}{a_n}.\]
Thus, every subshift over two symbols is of type $(\ls (g)\leq 1)$ for $\left(\left(\varphi(2^{-n})\right),P(X)\right)$. It implies that there is no result similar to Corollary \ref{genofthm2.8} for functions from $\G^0$ -- there exist aperiodic systems of type $(\ls (g)<\infty)$ for $((a_n),P(X))$, where  $g\in\G^0$ and $(a_n)$ has sublinear growth. 
Thus, the systems of type $(\ls=\infty)$ for $((a_n),P(X))$ can be distinguished by the $g$-entropy convergence types.

{\bf Choice of the sequence.} If $(X,T)$ has finite entropy and \mbox{$g\in\G^0$}, then for every finite partition $\mcp$~of $X$ we have\[\lim_{n\to\infty}\frac{H(g,\mcp_n)}{n}=0.\]
Thus, we will assume sublinear growth of the sequence $(a_n)$.

{\bf Choice of the function.} It is easy to see that if $g'(0)<\infty$, then every system is of type ($\Lig(g)=0$). Therefore we assume that $g'(0)=\infty$. Natural examples of such functions  are the following

\begin{equation}
\label{blumefunction0} g_0(x)=\begin{cases} x\log_a(1-\log_a x), & \text{for}\; x\in(0,1], \\ 0, & \text{for}\; x=0,\end{cases}
\end{equation}
or
\begin{equation}
\label{blumefunction1}
\tilde{g}_0(x)=\begin{cases}x (-\log_a x)^{\alpha}, &\text{for} \; x\in(0,1]\\ 0& \text{for} \;x=0,\end{cases}  
\end{equation}
with $a>1$ and $\alpha\in (0,1)$.\footnote{Moreover these functions are subderivative. In fact, if we want to obtain $g\in\G^0\cap\G'$ with $g'(0)=\infty$, it is sufficient to check whether $h(x)=a^xg(a^{-x})$ is concave, subadditive and increasing with sublinear growth. In this case for $g_0$~we have $h(x)=\log_a(1+x)$, and for $\tilde{g}_0$ it is $h(x)=x^{\alpha}$  $(\alpha \in(0,1))$, and both of them fullfill necessary conditions.}

\subsection{$g\in\G^{\infty}$}

In this case the analogues of Corollaries \ref{genofthm2.8}, \ref{genofthm3.9}, \ref{wn341tempa} and \ref{wn342tempa} (where instead of $\G^{\sh}$ we consider $\G^{\infty}$) hold. Every system of type $(\ls \geq c)$, $c\in (0,\infty]$  is of type \mbox{$(\ls(g)=\infty)$} for a given pair $((a_n),P)$.  On the other hand from Corollary \ref{genofthm2.8} we know that every ergodic system is not of type ($\ls <\infty$) for  $((a_n),P(X))$, $(a_n)$ with sublinear growth. Moreover systems of type $(\ls \leq c)$ are not frequent,  e.g. for invertible measure-preserving interval maps the set of such systems (for $(a_n)$ with sublinear growth)  is  I Baire category \cite[Thm 4.8]{Blume11.1}.

\begin{table}
\caption{\label{tab2} Connections between entropy and $g$-entropy types of convergence}
\begin{center}
\hspace{-0.5 in}
\begin{tabularx}{120mm}{|Y|Y|Y|Y|}
\hline
&$g\in\G^0$ & $g\in\G^{\sh}$ & $g\in\G^{\infty}$ \\ \hline
$\ls\leq c$ & $\ls(g)=0$ & $\ls(g)\leq \C(g)\cdot c$ & $\ls(g)\leq \infty$, $\ls(g)=\infty$ \\ \hline
$\ls\geq c$ & $\ls(g)=0$, $\ls(g)\geq 0$ & $\ls(g)\geq \C(g)\cdot c$ & $\ls(g)=\infty$ \\ \hline
$\ls< \infty$ & $\ls(g)=0$ & $\ls(g)< \infty$ & $\ls(g)\leq \infty$, $\ls(g)=\infty$\\ \hline
$\ls = \infty$ & $\ls(g)<\infty$, $\ls(g)\leq \infty$, $\ls(g) = \infty$  & $\ls(g)=\infty$ & $\ls(g)=\infty$ \\ \hline

\end{tabularx}
\end{center}
\end{table}
\subsection{Summary} We summarize our considerations in Tab. 1 (types $\ls$, $\ls(g)$ can be replaced by $\lig$ and $\lig(g)$ respectively). From these considerations it follows that the concept of generalized entropy convergence types is of use   for systems of type $(\ls\geq c)$ for $g\in\G^0$, and  of type $(\ls\leq c)$, when $g\in\G^{\infty}$. We will consentrate on  the first case, since the second, by the reasoning from Section 3.3, is less important.

\section{Case of $g\in\G^0$}

According to the discussion from the previous section the case of $g\in\G^0$ { needs more attention}. We assume that $g'(0)=\infty$. In this section we will show few results obtained using $g$-entropy convergence types  for $g\in\G^0$. 

Let $g$~be given by (\ref{blumefunction0}). For simplicity of computations we use  logarithm of base~2: \begin{equation}
\label{blumefunctionlog2} g_0(x):=x\log_2(1-\log_2 x), 
\end{equation}
 We know that $g_0\in\mathcal{G}_0^0\cap\G'$ and $g_0'(0)=\infty$ and $h(x)=2^xg_0(2^{-x})$, $x\in[0,\infty)$ is subadditive. This fact plays a~crucial rule in proofs performed in this note.

 \subsection{ Results for completely ergodic systems}
First, recall that every aperiodic transformation is isomorphic to some interval exchange map \cite{AOW}. Since every completely ergodic system is aperiodic, we may assume that the considered probability space is $([0,1],\Sigma_L,\mu_L)$, where $\mu_L$ is a~Lebesgue measure and $\Sigma_L$ is a~$\sigma$-algebra which consists of all Lebegue measurable subsets of the unit interval.
We show the analog of \cite[Thm 4.1]{Blume98} for $g$-entropy convergence rates. The proof is similar to one presented by Blume, but for the consistency of the note we give its proof.

\begin{theorem} \label{tw415}
If $\left([0,1],T\right)$ is completely ergodic, then there exists a~sequence $(a_n)$ with sublinear growth, such that for every $\mcp\in P([0,1])$ we have \[\liminf_{n\to\infty}\frac{H(g_0,\mcp_n)}{a_n}\geq 1.\]
\end{theorem}

 For $p,q\in\mathbb{N}$, $p\leq q$ let us define
\[P_{p}^{q}(n,E):=\{x\in X \;| \; p\leq p_n\left(s_{n}^E(x)\right)\leq q\}.\]

The crucial rule in the proof of Theorem \ref{tw415} will play the following lemma: %which gives the frequency of points of period ....
\begin{lemma}\label{ceuP122n}
Let $\left([0,1],T\right)$ be completely ergodic.  If $\varepsilon>0$, then there exists such a~(strictly) increasing sequence $\left(N_n\right)\subset \mathbb{N}$, that for every  $E\in\Sigma$ with $0<\mu(E)<1$, there exists $K$,~such that \[\mu\left(P_1^{2^{2^n}}\left(N_n,E\right)\right)<\varepsilon\]
for every $n\geq K$.
\end{lemma}
\begin{proof}
Let us define
\[\kappa_n:=\left\{\left[\frac{k-1}{2^{2^n}},\frac{k}{2^{2^n}}\right)\;|\; k\in\{1,\ldots,2^{2^n}\}\right\}\]
for $n\in\mathbb{N}$. For a given $n\in\mathbb{N}$, Ergodic Theorem and complete ergodicity of  $([0,1],T)$ imply that there exist  $M_n\in\mathbb{N}$,  $S_n\in\Sigma$ such that for every  $m\geq M_n$, $x\in S_n$ and $I\in \kappa_n$ we have
\begin{equation} \label{lem414w1}
\left|\frac 1m \sum_{k=0}^{m-1}{\bm 1}_{I}\left(\left(T^{2^{2^n}!}\right)^k(x)\right)-\mu(I)\right|<\frac{1}{n2^{2^n}}
\end{equation}
and
\begin{equation} \label{lem414w2}
\mu(S_n)>1-\frac 1n.
\end{equation}
Let \[N_n:=\left\lfloor\log_2\left(M_n\left(2^{2^n}!\right)\right)\right\rfloor+1.\]
Fix $\varepsilon>0$ and $E\in\Sigma$ is such that $\mu(E)\in (0,1)$. 
For  a~given $n\in\mathbb{N}$ we define set $F_n$ as:\[F_n:=\bigcup\{I\in\kappa_n\;|\; 2^{2^n}\mu(I\cap E)>1/2\}.\]
We set $\alpha:=\min\{\varepsilon/2, (1-\mu(E))/3,\mu(E)/3\}$. Lebesgue Theorem implies that there exists such $K>\alpha$, that for every $n\geq K$ occurs
\begin{equation} \label{lem414w3}
\mu(E\triangle F_n)<\alpha^2
\end{equation}
Inequality (\ref{lem414w1}) implies that 
\begin{equation}\label{lem414w4}
\left|\frac 1m \sum_{k=0}^{m-1}\bm{1}_{F_n}\left(\left(T^{2^{2^n}!}\right)^k(x)\right)-\mu(F_n)\right|<\frac{1}{n} \;\;\text{for}\;\;m\geq M_n \;\;\text{and}\;\; x\in S_n.
\end{equation}
For $n>0$ we define
\[R_n:=\left\{x\in[0,1]\left|\frac{1}{M_n}\sum_{k=0}^{m-1}\left|\bm{1}_{E}\left(\left(T^{2^{2^n}!}\right)^k(x)\right)-\bm{1}_{F_n}\left(\left(T^{2^{2^n}!}\right)^k(x)\right)\right|>\alpha\right.\right\}.\]
Then
\begin{eqnarray}
\alpha\mu(R_n)&\leq & \int_0^1 \frac{1}{M_n}\sum_{k=0}^{M_n-1}\left|\bm{1}_{E}\left(\left(T^{2^{2^n}!}\right)^k(x)\right)-\bm{1}_{F_n}\left(\left(T^{2^{2^n}!}\right)^k(x)\right)\right| d\mu \nonumber \\ &=&   \frac{1}{M_n}\sum_{k=0}^{M_n-1}\int_0^1\left|\bm{1}_{E}(x)-\bm{1}_{F_n}(x)\right| d\mu = \mu(E\triangle F_n).\nonumber
\end{eqnarray}
Thus, for $n\geq K$ we have \begin{equation}\label{lem414w5}
\mu(R_n)<\alpha.
\end{equation}
Suppose now that $x\in S_n\backslash R_n$ and $n\geq K$. Then from (\ref{lem414w3}), (\ref{lem414w4}), definition of $R_n$ and the choice of $K$ we obtain
\[\left|\frac{1}{M_n}\sum_{k=0}^{M_n-1}\bm{1}_{E}\left(\left(T^{2^{2^n}!}\right)^k(x)\right)-\mu(E)\right|\leq \left|\frac{1}{M_n}\sum_{k=0}^{M_n-1}\bm{1}_{F_n}\left(\left(T^{2^{2^n}!}\right)^k(x)\right)-\mu(F_n)\right|+\]
\[+\frac{1}{M_n}\sum_{k=0}^{M_n-1}\left|\bm{1}_{E}\left(\left(T^{2^{2^n}!}\right)^k(x)\right)-\bm{1}_{F_n}\left(\left(T^{2^{2^n}!}\right)^k(x)\right)\right|+|\mu(F_n)-\mu(E)|<\frac 1n+2\alpha<3\alpha.\]
If $x\in P_1^{2^{2^n}}(2^{N_n},E)$, then  $2^{2^n}!$ is divisible by $p_{2^{N_n}}(x)$, and because \mbox{$M_n(2^{2^n}!)<2^{2^{N_n}}$} it is easy to see that \[\frac{1}{M_n}\sum_{k=0}^{M_n-1}\bm{1}_{E}\left(\left(T^{2^{2^n}!}\right)^k(x)\right)=\left\{\begin{array}{ll} 0,&\;\;\text{for}\;\;x\in[0,1]\backslash E,\\ 1,&\;\;\text{for}\;\;x\in E. \end{array}\right.\]
Therefore \[\left|\frac{1}{M_n}\sum_{k=0}^{M_n-1}\bm{1}_{E}\left(\left(T^{2^{2^n}!}\right)^k(x)\right)-\mu(E)\right|\geq \min\{\mu(E),1-\mu(E)\}\geq 3\alpha\]
and for $n\geq K$ we have $P_1^{2^{2^n}}(2^{N_n},E)\subset [0,1]\backslash (S_n\backslash R_n)$. Hence \[\mu(P_1^{2^{2^n}}(2^{N_n},E))\leq 1-\mu(S_n\backslash R_n)<2\alpha\leq \varepsilon.\]
We may choose such $(N_n)$, that it is strictly increasing. This completes the proof.
\end{proof}
\begin{proof}[Proof of Theorem \ref{tw415}]
Lemma \ref{ceuP122n} applied to $\varepsilon=1/2$ implies that there exists a~sequence $(N_n)$ as in the statement of Lemma \ref{ceuP122n}.
Define $(a_k)$ as follows:
if $1\leq k\leq 2^{2^{N_1}}$, then $a_k:=1$ and if $k\in\{2^{2^{N_n}},\ldots,2^{2^{N_{n+1}}-1}\}$, then $a_k:=n$.  The definition implies clearly that $(a_k)$ is increasing to the infinity. Let $E\in P(X)$. If $s\in\{0,1\}^{2^{2^n}}$ is such that $p_{2^{N_n}}(s)\geq 2^{2^n}$, then by Lemma \ref{lemma16Blume} we have\[\mu(A_{2^{N_n}}^E(s))<2^{-2^n}.\] Thus, $\mcp_{2^{2^{N_n}}}^E$ induces a~partition on $P_n:=P_{2^{2^n}+1}^{2^{2^{N_n}}}(N_n,E)$ such that for \mbox{$A\in \mcp_{2^{2^{N_n}}}^E$} we have \[\mu(A\cap P_n)<2^{-2^n}.\]
Hence, applying Lemmas \ref{lemmaH-Hzaw}, \ref{lemlambda}, we obtain
\[H(g_0,\mcp_{2^{2^{N_n}}}^E)\geq H_{P_n}(g_0,\mcp_{2^{2^{N_n}}}^E)-2 \geq \mu(P_n)\log_2(2^n+1)-2.\]
From the definition of $(a_k)$ we get for all $k\in\{2^{2^{N_n}},\ldots,2^{2^{N_{n+1}}}-1\}$ that \[\frac{H(g_0,\mcp_k^E)}{a_k}\geq \frac{H(g_0,\mcp_{2^{2^{N_n}}}^E)}{n}\geq \frac{\mu(P_n)\log_2(2^n+1)-2}{n}.\]
Lemma \ref{ceuP122n} assures that $\liminf\limits_{n\to\infty}\mu(P_n)\geq 1/2$ and \[\liminf_{k\to\infty}\frac{H(g_0,\mcp_k^E)}{a_k}\geq \liminf_{n\to\infty} \mu(P_n)\cdot \frac{\log_2(2^n+1)}{n}-\frac 2n\geq \frac 12.\]
Replacing $a_k$ by $a_k/2$ gives us the desired conclusion.
\end{proof}

Theorem \ref{tw415} implies the following corollary
\begin{corollary} \label{wn49}
Under the assumptions of Theorem \ref{tw415} for every \mbox{$\mcp\in P([0,1])$} we have \[\lim_{n\to\infty} H(\mcp_n)=\infty.\]
In other words, for every completely ergodic system $(X,T)$ there exists $(a_n)$  with sublinear growth, such that $(X,T)$ is of type $(\Lig =\infty)$ for $((a_n),P(X))$.
\end{corollary}
\begin{proof}
Let $(a_n)$ be the sequence from the proof of Theorem \ref{tw415}. Then \[\liminf_{n\to \infty}\frac{H(\mcp_n)}{a_n}=\liminf_{n\to\infty}\frac{H(\mcp_n)}{H(g_0,\mcp_n)}\cdot \frac{H(g_0,\mcp_n)}{a_n}=\infty,\] which completes the proof.
\end{proof}

Note that every weakly mixing system is completely ergodic so the above statement is true, for example, for any weakly mixing system. On the other hand, the claim is not true without the assumption of complete ergodicity of $(X,T)$, because then the system $([0,1], T) $ can have e.g. a~periodic factor (see \cite[Remark 4.2] {Blume98}).

 \subsection{Results for aperiodic systems} In this subsection instead of the class  $P(X)$ we consider the class $R(X)$ defined as in (\ref{Rx}). If $(X,T)$ is a~measure-preserving system with $R(X)\neq \emptyset$, then it is aperiodic \cite{Blume97}. $R(X)$ is still ``sufficiently large''  in a sense that each type of convergence considered on $R(X)$ is an isomorphism invariant. On the other hand we choose this class because we want to exclude aperiodic systems for which there are  $\mcp\in P(X)$, such that there exists $N\in\mathbb{N}$ with $\mcp_n=\mcp_N$ for $n\geq N$.
Consider \[g_m(x)=xh^{(m+1)}(-\log_2 x),\;\; \text{with}\;\; h(x)=\log_2(1+x),\] where $m\in\mathbb{N}$, and \[h^{(m)}=\underbrace{h\circ\cdots\circ h}_{m \;\text{times}}.\]Repeating the reasoning from  \cite[Sec. 3]{Blume97} replacing  $(2^i)$ by $\left(2^{2^{\cdot^{\cdot^{\cdot{2^i}}}}}\right)$ (where we iterate $i\mapsto 2^i$, $m+1$ times)  we obtain the following theorem which is an anolog of \cite[Thm 3.9]{Blume97} for $g_m$: 

\begin{theorem}\label{aperg0}
If $(X,T)$ is aperiodic and measure-preserving and $\phi\colon [0,\infty)\mapsto (0,\infty)$ is an increasing function with \[\int\limits_1^{\infty}\frac{\phi(x)}{x^2}dx<\infty,\] then for every  $\mcp\in R(X)$ we have
\[\limsup_{n\to\infty}\frac{H(g_m,\mcp_n)}{\phi(\log_2^{(m+2)} n)}=\infty.\]
\end{theorem}
Blume pointed out the existence of this theorem for $g_0$. This claim is essentially the application of his observation. Therefore, we do not present the proof of this fact.

\section{Type $(LI(g)\geq 0)$ as an isomoprhism invariant for weakly mixing rank one transformations}

In this section we will show how we can use generalized entropy convergence rates to prove that two systems are non-isomorphic. At this purpose we remind the class of weakly mixing, rank one systems  introduced by Blume in \cite{Blumenotpubl}. The dynamics of these systems may be (due to weakly mixing property) quite complicated.
At the same time, they are generated by the cutting and stacking process, which allows us to control the growth rate of  $(H(g,\mcp_n))$ for every partition $\mcp$.
It appears that if we use other functions than the Shannon function $\eta$~we will be able to expand Blume's results.
We show that if one chooses the appropriate function $ g\in\G^0$, one can obtain theorem similar to \cite[Thm 4.22]{Blumenotpubl}, which allows us to distinguish systems that do not meet the assumptions of  Blume's theorem. Additionally we will fill the gap in the original proof.

\begin{remark}
The constructed class of weakly mixing rank one systems  will be parameterized by sequences of prime numbers  $\psi=(p_n)$, for which $\sum\limits_{n=0}^{\infty} p_n^2/p_{n+1}<\infty$. 
To each such a~sequence $\psi$ we assign a~weakly mixing, rank one  transformation $T$  and the interval $J\subset \mathbb{R}$.
Since $\sum\limits_{n=0}^{\infty}p_n^2/p_{n+1}$ is finite, for simplicity of computations we will always assume that $(6p_n^2+1)/p_{n+1}<1$ for $n\in\mathbb{N}$.
\end{remark}

\subsection{Construction of $\Gamma$}
%\subsubsection{konstrukcja - podejście geometryczne}
Let $\xi=(p_n)$ be such a~sequence of primary numbers that
\begin{equation}\label{sumpn}
\sum_{n=0}^{\infty} \frac{p_n^2}{p_{n+1}}<\infty \;\; \textrm{and}\;\; \frac{6p_n^2+1}{p_{n+1}}<1 \; \textrm{for}\; n\in\mathbb{N}.
\end{equation}

 We will use $\xi$ to define a~sequence of towers $(\tau_n)$ of height $2p_n^2$  and an increasing sequence of positive numbers $(x_n)$, which we will use to define $J_{\xi}$. The construction will be given recursively.
Let $\tau_0:=\left([0,1),[1,2),\ldots,[2p_0^2-1,2p_0^2)\right)$,  $x_0:=2p_0^2$  and let $T_{\tau_0}\colon [0,2p_0^2-1)\mapsto [1,2p_0^2)$ assign points from a~given level of the tower $\tau_0$ to the points from the  next level of the tower (it is defined everywhere except  the highest level of the tower -- $I_{2p_{0}^2-1}^{(0)}$). Assume that $\tau_{n-1}$ and $x_{n-1}$ are defined and $\tau_{n-1}=(I_0^{(n-1)},\ldots,I_{2p_{n-1}^2-1}^{(n-1)})$ is a~tower (consisting of intervals),  for which $\bm{\tau}_{n-1}=[0,x_{n-1})$ and  $T_{\tau_{n-1}}$ is a~transformation, which  assigns points from a~given level of the tower $\tau_{n-1}$ to the points from the next level of the tower (it is defined everywhere except  the highest level of the tower  -- $I_{2p_{n-1}^2-1}^{(n-1)}$). 

\begin{itemize}
\item[{\bf Step 1}.] Let
\[y_n:=x_{n-1}+\frac 13 \mu_L(I_0^{n-1}),\;\;\]
\[k_n:=\left\lfloor\frac{p_n}{6p_{n-1}^2+1}\right\rfloor,\;\; j_n:=p_n-k_n\cdot(6p_{n-1}^2+1)\; \;\textrm{and}\;\; x_n:=y_n+j_n\cdot\frac{\mu_L(I_0^{(n-1)})}{3k_n}\]
where $\mu_L$ is a Lebesgue measure.

\begin{figure}[ht]
    \centering
  \includegraphics[width=0.53\textwidth, angle=0]{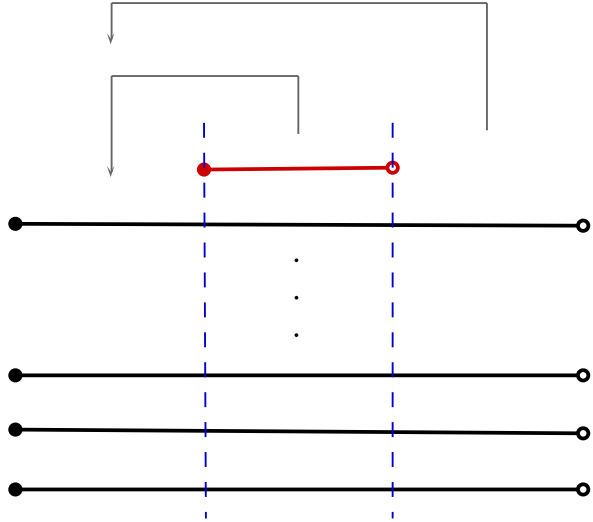}
\caption{ \label{krok1}
Step 2. -- cutting and stacking of the tower $\tau_{n-1}$ with the spacer $[x_{n-1},y_n)$ (red line)}
\end{figure}

\item[{\bf Step 2}.] Consider a~tower $\tau_{n-1}$, it has $2p_{n-1}^2$ levels, of measure $\mu_L\left(I_0^{(n-1)}\right)$ each. Over the highest level we  put a~spacer $[x_{n-1},y_n)$ (of measure $\mu_L\left(I_0^{(n-1)}\right)/3$). Then we cut the tower into three subtowers and stack them (the first at the bottom, the second (with a~spacer) over it, and the third one at the top). We obtain a~tower of height  $6p_{n-1}^2+1$, which every level has measure $\mu_L\left(I_0^{(n-1)}\right)/3$ (see Fig. \ref{krok1}).%\footnote{jeden czerwony kawałek ;)}

\item[\bf{Step 3}.] We cut the tower obtained in the previous step vertically into $k_n$ subtowers, each of measure $\mu_L\left(I_0^{(n-1)}\right)/3k_n$.%\footnote{$k_n$. czerwonych poziomów -- już nieco przemieszanych}
\item[\bf{Step 4}.] We cut an interval $[y_n,x_n)$ into $j_n$ subintervals of measure $\mu_L\left(I_0^{(n-1)}\right)/3k_n$ and stack them (as in Step 2.). Then we put this tower over  the tower obtained in Step 3. and get a~tower of height $p_n$, where each level is of measure $\mu_L\left(I_0^{(n-1)}\right)/3k_n$ (see Fig. \ref{krok2}).

\begin{figure}[ht]
    \centering
  \includegraphics[width=0.53\textwidth, angle=0]{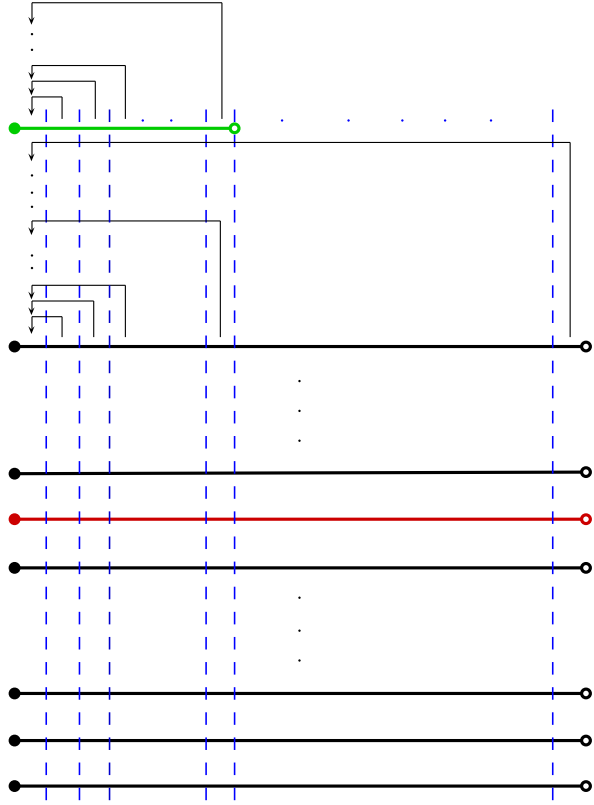}
\caption{ \label{krok2}
Steps 3. and 4. of construction of $\tau_n$}
\end{figure}

\item[{\bf Step 5}.] We cut the tower obtained in Step 4. into  $p_n$ subtowers of the same length and stack them. This tower has  $p_{n}^2$ levels.
\item[{\bf Step 6}.] We cut the tower from the previous step into two subtowers and we stack them. 
\end{itemize}
We denote the obtained tower by $\tau_n$, and its levels by $I_j^{(n)}$ for $j\in\{0,\ldots,2p_n^2-1\}$.
Moreover $\bm{\tau}_n=[0,x_n)$ and the transformation  $T_{\tau_n}$ is an expansion of $T_{\tau_{n-1}}$ onto the highest level of $\tau_{n-1}$ and all spacers but the last subinterval  of $[y_n,x_n)$, i.e. $[x_n-\frac{x_n-y_n}{j_n},x_n)$, which is the highest level of $\tau_n$. Since $\sum\limits_{n=0}^{\infty} p_n^2/p_{n+1}<\infty $, there exists a~finite limit of  $(x_n)$ \cite{Blumenotpubl}. Let $J_{\xi}:=[0,\lim\limits_{n\to\infty}x_n)$.
 We know that $\bigcup\limits_{n=0}^{\infty} \bm{\tau}_n=J_{\xi}$. We define \mbox{$T_{\xi}:J_{\xi}\to J_{\xi}$} by
\[T_{\xi}(x):=\lim_{n\to\infty} T_{\tau_n}(x)\;\;\textrm{for}\;\; x\in J_{\xi}.\]
Considering a~$\sigma$-algebra of Lebesgue measurable subsets of $J_{\xi}$ with a measure \[\mu_{\xi}:=\frac{\mu_L}{\mu_L(J_{\xi})}\] we obtain the dynamical system  $(J_{\xi}, \Sigma_{\xi},\mu_{\xi},T_{\xi})$. With this system we associate a~class $P(J_{\xi})$.
 The transformation $T_{\xi}$ is rank one, therefore it has zero Kolmogorov-Sinai entropy. Moreover it is weakly mixing.\footnote{The proof of weakly mixing follows the same argument as the proof of weakly mixing of Chacon transformation \cite[Ch. 6.5]{Silva}.}
We define a class of systems

\[\Gamma:=\left\{(J_{\xi},T_{\xi})\;\left|\;\xi=(p_n),\;\text{where}\; p_n \in\mathbb{P}, \;\sum_{n=1}^{\infty}\frac{p_{n-1}^2}{p_n}<\infty \; \textrm{and} \; \frac{6p_{n-1}^2+1}{p_n}<1\right.\right\}.\]

Our goal is to use the $\lig(g)\geq c$ convergence type in order to decide whether two systems in $\Gamma$ are non-isomorphic. The fact that all systems in $\Gamma$ are weakly mixing shows that they cannot easily be distinguished according to their spectral properties. This suggests that the rate of entropy convergence might be a~useful isomorphism invariant. Otherwise the weak mixing property is irrelevant for our discussion.

\subsection{Choice of $(a_n)$}
To understand how we can show that a~given system  $(X,T)$ is of type $(\lig(g)\geq c)$, assume that $(a_n)$ is an increasing sequence of positive numbers converging to the infinity, such that $(X, T)$ is of type
$(\ls(g) \geq c)$ for $((a_n),P (X))$. Then for every \mbox{$\mcp\in P(X)$} there exists a~strictly increasing sequence $\nu=(n_k)\subset \mathbb{N}$ such that \[\liminf_{k\to\infty} \frac{H(g,\mcp_{n_k})}{a_{n_k}} \geq c.\] In general, $\nu$ is dependent on a~partition $\mcp$, but if $\nu$ is independent of the choice of $\mcp$, then  $(X, T)$ is of type $(\lig(g)\geq c)$ for $((\nu(a_n)), P (X))$ (see Lemma \ref{lemmaliminfnu}), where
\[ \nu (a_n):=a_{n_k} \;\; \textrm{for}\;\; n_k \leq n<n_{k+1}.\]

Properties of $(a_n)$ imply that $\left(\nu (a_n)\right)$ is increasing and $\lim\limits_{n\to\infty}\nu(a_n)=\infty$. %(We notice here that for the definition of ν (an) it is of %course not important that ν ⊂ N. We just need ν to be a %sequence which strictly increases to ∞. %We will make use of this generalized definition in %Theorem 1.20.)

\begin{lemma}
\label{lemmaliminfnu}
Let $(X,T)$ be a~measure-preserving system, $c>0$, $\nu=(n_k)\subset\mathbb{N}$ be strictly increasing and $(a_n)\subset\mathbb{N}$. If \[\liminf_{k\to\infty}\frac{H(g,\mcp_{n_k})}{a_{n_k}}\geq c\;\; \textrm{for every} \;\; \mcp\in P(X),\]
then
\[\liminf_{n\to\infty}\frac{H(g,\mcp_{n})}{\nu(a_{n})}\geq c\;\; \textrm{for every} \;\; \mcp\in P(X).\]
In other words,  $(X,T)$ is of type  $(\lig(g)\geq c)$ for $((\nu(a_n)),P(X))$.
\end{lemma}

\begin{proof}
This fact follows from the monotonicity of $(H(g,\mcp_n))$ for a~given \mbox{$\mcp\in P(X)$.} Let $k=k(n)$. Then
\[\frac{H(g,\mcp_n)}{\nu(a_n)}=\frac{H(g,\mcp_n)}{a_{n_{k(n)}}}\geq \frac{H\left(g,\mcp_{n_{k(n)}}\right)}{a_{n_{k(n)}}}.\]
Converging with $n$~(and hence with $k(n)$) to infinity we obtain the assertion.
\end{proof}

\subsection{Choice of $g\in \G$}
We want to use the introduced invariant to check whether two systems are isomorphic. Our aim will be to  find an analogue of \cite[Thm 4.22]{Blumenotpubl} which may by used in the case, when the Blume's theorem fails. We will use a~concave, increasing (to plus infinity) function $h\colon [0,\infty)\mapsto\mathbb{R}$ with sublinear growth, for which there exist sequences $(a_n),(b_n),(c_n),(d_n)$, each of them converging to the infinity (with $n\to\infty$), with
\begin{equation} \label{hanbn} 0<\lim_{n\to\infty}\frac{a_n}{b_n}= \lim_{n\to\infty}\frac{c_n}{d_n}<\infty,
\;\;\;\text{and}\;\;\;
\lim_{n\to\infty}\frac{h(a_n)}{h(b_n)}\neq \lim_{n\to\infty}\frac{h(c_n)}{h(d_n)}.\end{equation}
Functions, which fullfill this condition are called {\sf not regularly varying}. We choose such a~function, since for systems from $\Gamma$ we should have a~significantly different condition than the one in \cite[Thm 4.22]{Blumenotpubl}.  The function which we will use in this section was proposed by Iksanow and R\"{o}sler \cite{Rosler}:\[h(x):=\left\{\begin{array}{ll}x,& \text{for}\;\; x\in[0,1)\\ 2^{-k}x+2^{k+1}-2,& \text{for}\;\; x\in[4^k,4^{k+1}) ,\;\;k=0,1,\ldots\end{array}\right.\]
It is a~concave, subadditive function with sublinear growth. Defining\label{hwlasn}
 \[g(x):=x \cdot h(-\log_2 x)\] we obtain a~function from the set $\G^0\cap\G'$ with infinite derivative at zero, given by
\[
g(x)=\left\{\begin{array}{ll}0,&\text{for}\;\; x=0,\\ -2^{-k}x\log_2x+x(2^{k+1}-2),&\text{for}\;\; x\in\left(2^{-4^{k+1}},2^{-4^k}\right],\;\;k=0,1,\ldots\\ -x\log_2 x,& \text{for}\;\; x\in\left(\frac 12,1\right].\end{array}\right.
\]
   
\subsection{Auxiliary theorems}
Types $\lig(g)$ fulfill the following theorem for systems from $\Gamma$:
\begin{theorem} \label{thm-liggeq} 
If $(J_{\xi},T_{\xi})\in \Gamma$ and $E\in \Sigma_{\xi}$ is such that $0<\mu_{\xi}(E)<1$, then
\begin{equation} \label{liminf-thm}
\liminf_{n\to\infty}\frac{H\left(g,\mcp_{2p_n^2}\right)}{h\left(\log_2 2p_n^2\right)}\geq \frac 14.
\end{equation}
\end{theorem}
To prove this theorem we will need estimations of values of the sequence $(H(g,\mcp_n))$ for an arbitrary $\mcp\in P(X)$. We will use the symbolic representation of atoms from $\mcp_n$.  We state the following fact, which comes from the proof of \cite[Thm 4.18]{Blumenotpubl}:
\begin{fact}\label{lemQn}
Let $\psi\colon [0,1]\mapsto \mathbb{R}$ be given as $\psi(x):=2x\cdot(1-x)$ and $E\in\Sigma_{\xi}$ is such that $\mcp^E\in P(J_{\xi})$. Define  $\lambda_E:=\psi(\mu_{\xi}(E))/8$ and
$Q_n:=P_1^{\left\lfloor\lambda_Ep_n\right\rfloor}\left(2p_n^2,E\right)$.  Then \[\limsup_{n\to\infty}\frac{\mu_{\xi}\left(Q_n\cap\bm{\tau}_n\right)}{\mu_{\xi}(\bm{\tau}_n)}\leq \frac 12.\]
\end{fact}

\begin{proof}[Proof of Theorem \ref{thm-liggeq}] Let $\psi$, $\lambda_E$ and $Q_n$ be given as in Fact \ref{lemQn}.
Let $R_n:=\bm{\tau}_n\backslash Q_n$. Fact \ref{lemQn} implies that   \[\liminf_{n\to\infty}\mu_{\xi}(R_n)=\liminf_{n\to\infty}\frac{\mu_{\xi}(R_n)}{\mu_{\xi}(\bm{\tau}_n)}\geq \frac 12.\] From the definition of $R_n$, we know that \[p_{2p_n^2}(s_{2p_n^2}^E(x))>\lfloor\lambda_Ep_n\rfloor\; \text{for all} \; x\in R_n.\] Therefore, Lemma \ref{lemma16Blume} implies that  for $x\in R_n$ we have  \[\mu_{\xi}\left(A_{2p_n^2}^E(s_{2p_n^2}^E(x))\cap R_n\right)\leq \mu_{\xi}\left(A_{2p_n^2}^E(s_{2p_n^2}^E(x))\cap \bm{\tau}_n\right)<\frac{2}{\lfloor\lambda_Ep_n\rfloor}\mu_{\xi}(\bm{\tau}_n)<\frac{2}{\lfloor\lambda_Ep_n\rfloor}.\]
Applying  Lemmas \ref{lemmaH-Hzaw} and  \ref{lemlambda} we obtain that
\[H(g,\mcp_{2p_n^2})  \geq \mu_{\xi}(R_n)h\left(\log_2\left(\lfloor\lambda_E p_n\rfloor/2\right)\right)-2.\]
Thus,
\begin{eqnarray}
\liminf_{n\to\infty}\frac{H(g,\mcp_{2p_n^2})}{h(\log_2 2p_n^2)}&\geq & \frac 12\liminf_{n\to\infty}\frac{h(\log_2(\lfloor\lambda_Ep_n\rfloor/2))}{h(\log_22p_n^2)}= \frac 12\liminf_{n\to\infty}\frac{h(\log_2(\lambda_Ep_n))}{h(2\log_2p_n)}\nonumber \\ &\geq & \frac 14 \liminf_{n\to\infty}\frac{h(\log_2(\lambda_Ep_n))}{h(\log_2p_n)}.\nonumber
\end{eqnarray}
It is sufficient to show that \begin{equation} \label{aaaaa}\liminf_{n\to\infty}\frac{h(\log_2(\lambda_Ep_n))}{h(\log_2p_n)}\geq 1.\end{equation}
Let  $m_n$ be such that $\lambda_Ep_n \in \left[2^{4^{m_n}},2^{4^{m_n+1}}\right)$. Then for sufficiently large $n$ we have $p_n \in \left[2^{4^{m_n}},2^{4^{m_n+2}}\right)$. Thus, if we want to find the lower limit of the quotient $ H\left(g,\mcp_{2p_n^2}\right)/h(\log_22p_n^2)$, we have to consider the following cases:\\
{\bf Case 1}. If $p_n\in \left[2^{ 4^{m_n}},2^{4^{m_n+1}}\right)$, then
\[\liminf_{n\to\infty}\frac{h(\log_2(\lambda_Ep_n))}{h(\log_2p_n)}=\liminf_{n\to\infty}\frac{h(\log_2p_n)}{h(\log_2p_n)}=1.\]
{\bf Case 2}. If $p_n\in \left[2^{ 4^{m_n+1}},2^{ 4^{m_n+2}}\right)$, then
\begin{eqnarray}\liminf_{n\to\infty}\frac{h(\log_2(\lambda_Ep_n))}{h(\log_2p_n)}&=&\liminf_{n\to\infty}\frac{2^{-m_n}\log_2(\lambda_Ep_n)+2^{m_n+1}-2}{2^{-m_n-1}\log_2p_n+2^{m_n+2}-2}\nonumber\\ &=& \liminf_{n\to\infty}\frac{2^{-m_n}\log_2p_n+2^{m_n+1}}{2^{-m_n-1}\log_2p_n+2^{m_n+2}}\geq 1,\nonumber\end{eqnarray}
where the last inequality comes from the fact that $f\colon(0,\infty)\mapsto \mathbb{R}$ given as \begin{equation} \label{fxrosnaca} f(x):=\frac{2^{-m_n}x+2^{m_n+1}}{2^{-m_n-1}x+2^{m_n+2}}\end{equation} is increasing in $\left[2^{ 4^{m_n+1}},2^{ 4^{m_n+2}}\right)$ for any $m_n$ and $f(2^{4^{m_n+1}})=1$.
This implies (\ref{aaaaa}) and completes the proof.
\end{proof}

\begin{corollary}
Every $(J_{\xi},T_{\xi})\in\Gamma$ is of type $(\ls(g)\geq 1/4)$ for $((h(\log_2n)),P(J_{\xi}))$.
\end{corollary}

\begin{proof} For every $E\in\Sigma_{\xi}$ such that $0<\mu_{\xi}(E)<1$ we have
\[\limsup_{n\to\infty}\frac{H(g,\mcp_n^E)}{h(\log_2n)}\geq \liminf_{n\to\infty}\frac{H\left(g,\mcp_{2p_n^2}\right)}{h\left(\log_22p_n^2\right)}\geq \frac 14, \]
which completes the proof.
\end{proof}

Theorem \ref{thm-liggeq} and Lemma \ref{lemmaliminfnu} imply the following corollary:
\begin{corollary}\label{CorLIG14}
If $(J_{\xi},T_{\xi})\in\Gamma$ and $\zeta:=(2p_n^2)$, then $(J_{\xi},T_{\xi})$ is of type $(\lig(g)\geq \frac 14)$ for $\left(\left(\zeta \left(h(\log_2n)\right)\right),P(J_{\xi})\right)$.
\end{corollary}

We will use this corollary to distinguish systems  from $\Gamma$. 

\subsection{Main theorem}
Let us define a family of systems
\[\Gamma_0:=\left\{(J_{\xi},T_{\xi})\in\Gamma\;\left|\;\xi=(p_n)\; \text{such that}\;\left(\frac{p_{n-1}^3}{p_n}\right) \; \text{is bounded} \right. \right\}.\]

Let $(J_{\xi_0},T_{\xi_0})\in\Gamma_0$ and $(J_{\xi},T_{\xi})\in\Gamma$ for $\xi_0=(p_n)$, $\xi=(q_n)$ and $\zeta=(2q_n^2)$. Corollary \ref{CorLIG14} implies that  $(J_{\xi},T_{\xi})$ is of type $(\lig(g)\geq 1/4)$ for $\left(\left(\zeta \left(h(\log_2n)\right)\right),P(J_{\xi})\right)$. Under some additional conditions on $(J_{\xi},T_{\xi})$ and $(J_{\xi_0},T_{\xi_0})$, we will show that $(J_{\xi_0},T_{\xi_0})$ is not of type $(\lig(g)\geq  1/4)$ for $\left(\left(\zeta \left(h(\log_2n)\right)\right),P(J_{\xi})\right)$, which implies that $\xi$ and $\xi_0$ are not isomorphic. To show that $(J_{\xi_0},T_{\xi_0})$ is not of type $(\lig(g)\geq  1/4)$ for $\left(\left(\zeta \left(h(\log_2n)\right)\right),P(J_{\xi})\right)$ it is sufficient to find such $E\in\Sigma_{\xi_0}$, that
\[\liminf_{n\to\infty}\frac{H(g,\mcp_n^E)}{\zeta(h(\log_2n))}<\frac 14.\]

The following theorem allows us to distinguish systems from $\Gamma$ and $\Gamma_0$:
\begin{theorem} \label{gamma0nonisom}  
Let $(J_{\xi_0},T_{\xi_0})\in \Gamma_0$, where $\xi_0=(p_n)$ is such that there exists $r\in\mathbb{N}$, for which for sufficiently large $n$ we have \begin{equation}\label{pnr}
p_n<p_{n-1}^{2r}.
\end{equation}  Let $(J_{\xi},T_{\xi})\in \Gamma$, where $\xi=(q_n)$. Let $a,b>0$, $a+b<1/4$ and define $a\xi_0:=(ap_n)$, $\zeta:=(2q_n^2)$. If
\begin{equation}
\label{gamma0nonisome1} 
\liminf_{n\to\infty} \frac{a\xi_0 \left(h(\log_2n) \right)}{\zeta \left(h(\log_2n) \right)}<\frac{b}{2r},
\end{equation}
then $(J_{\xi_0},T_{\xi_0})\in \Gamma_0$ and $(J_{\xi},T_{\xi})\in \Gamma$ are not isomorphic.
\end{theorem}

It is an analogue of  \cite[Thm 4.22]{Blumenotpubl} for types of Shannon entropy convergence rates, where instead of  (\ref{gamma0nonisome1}) it is assumed that
\begin{equation} \label{ababab}\liminf_{n\to\infty} \frac{a\xi_0 \left(\log_2 n \right)}{\zeta \left(\log_2 n \right)}<\frac b2
\end{equation} and the assumption (\ref{pnr}) is missing.\footnote{In fact in the Blume's proof there is a~gap. More specifically the control of the growth of  $(p_n)$ is needed, since $g$-entropies should not grow to fast.  However, repeating the proof presented below for $\eta$ one can  obtain the following revised version of \cite[Thm 4.22]{Blumenotpubl}:
\begin{theorem} \label{gamma0nonisomshannon}  
Let $(J_{\xi_0},T_{\xi_0})\in \Gamma_0$, where $\xi_0=(p_n)$ is such that there exists $r\in\mathbb{N}$ for which for sufficiently large $n$~we have $p_n<p_{n-1}^{2r}$.  Let $(J_{\xi},T_{\xi})\in \Gamma$,  $\xi=(q_n)$. Let $a,b>0$, $a+b< 1/4$ and $\zeta:=(2q_n^2)$. If
\begin{equation} \label{bbbbbbbbbbb}
\liminf_{n\to\infty} \frac{a\xi_0 \left(\log_2n \right)}{\zeta \left(\log_2n\right)}<\frac{b}{2r},
\end{equation}
then $(J_{\xi_0},T_{\xi_0})\in \Gamma_0$ and  $(J_{\xi},T_{\xi})\in \Gamma$ are not isomorphic.
\end{theorem}}
The proof in general follows steps of the proof of \cite[Thm 4.22]{Blumenotpubl}.
Moreover due to the fact that $h$ is not regularly varying, we might expect that there exist systems indistinguishable by \cite[Thm 4.22]{Blumenotpubl} but distinguishable by Theorem \ref{gamma0nonisom} (and vice versa).

Before we prove this theorem we state the following technical lemma, which will allow us to estimate $g$-entropy of a~partition $\mcp_k^{[0,1)}$.

\begin{lemma}\label{Lemma1.19}
Let $(J_{\xi_0},T_{\xi_0})\in\Gamma_0$ fulfill the assumptions of Theorem \ref{gamma0nonisom}. Let $\varepsilon>0$, $0<a<1/4$ and \[d_{\xi_0}:=\left(6\sup_{n\geq 0}\frac{p_{n-1}^3}{p_n}+2\right)\left(1+\sum_{n=1}^{\infty}\frac{p_{n-1}^2}{p_n}\right)+10,\;\;\text{where}\;\; \xi_0=(p_n).\]
Then there exists $N\in\mathbb{N}$ such that for every $n\geq N$, we have $ap_n>(6p_{n-1}^2+1)^{5/4}$,
and \begin{eqnarray} H\left(g,\mcp_k^{[0,1)}\right)&\leq &r\left(2^{-m_{n,1}}\log_26p_{n-1}^2+2^{m_{n,1}+1}\right)\nonumber +d_{\xi_0}+2\\&+& \frac{k}{p_n}(1+\varepsilon)\left[2^{-m_{n,2}}\log_2 k+2^{-m_{n,2}}+2^{m_{n,2}+1}\right]  \nonumber \end{eqnarray} for every $k\in\left[(6p_{n-1}^2+1)^{5/4},ap_n\right]\cap\mathbb{N}$, where $m_{n,1}$, $m_{n,2}$ are such that \[6p_{n-1}^2\in\left[2^{4^{m_{n,1}}},2^{4^{m_{n,1}+1}}\right)\;\; \text{and}\;\; 12p_{n-1}^2+k\in\left[2^{4^{m_{n,2}}},2^{4^{m_{n,2}+1}}\right).\]
\end{lemma} 
We can interprete the above lemma  as follows: if $6p_{n-1}^2+1<k\ll p_n$, then $\zeta (h(\log_2k))=h(\log_22p_{n-1}^2)$, thus, Lemma  \ref{Lemma1.19} implies that \[H\left(g,\mcp_k^{[0,1)}\right)\approx \zeta(h(\log_2k)).\]
First we will use this lemma in the proof of Theorem \ref{gamma0nonisom} and then we will complete the proof showing Lemma \ref{Lemma1.19}.
\begin{proof}[Proof of Theorem \ref{gamma0nonisom}]
Corollary  \ref{CorLIG14} implies that $(J_{\xi},T_{\xi})$ is of type $\left(\lig(g)\geq 1/4\right)$ for $\left(\left(\zeta \left(h(\log_2n)\right)\right),P(J_{\xi})\right)$. It is sufficient to show that  $(J_{\xi_0},T_{\xi_0})$ is not of type $\left(\lig(g)\geq 1/4\right)$ for $\left(\left(\zeta \left(h(\log_2n)\right)\right),P(J_{\xi_0})\right)$.

The condition (\ref{gamma0nonisome1}) implies that there exist $\delta>0$ and strictly increasing sequences $(l_i),\;(m_i),\; (n_i)$ such that for every $i\in\mathbb{N}$ we have
\begin{equation}
\label{gamma0nonisome2}
ap_{n_i-1}\leq l_i <ap_{n_i},
\end{equation}
\begin{equation}
\label{gamma0nonisome3}
2q_{m_i-1}^2\leq l_i <2q_{m_i}^2,
\end{equation}
\begin{equation}
\label{gamma0nonisome4}
\frac{a\xi_0\left(h\left(\log_2l_i\right)\right)}{\zeta \left(h(\log_2l_i)\right)}<\frac{b}{2r}(1-\delta).
\end{equation}

Therefore  (\ref{gamma0nonisome2}) and (\ref{gamma0nonisome3})  imply that for every $i\geq 1$ we have
\begin{equation}
\label{gamma0nonisome5}
a\xi_0\left(h(\log_2l_i)\right)=h\left(\log_2 ap_{n_i-1}\right)=2^{-t_i} \log_2 ap_{n_i-1} +2^{t_i+1}-2
\end{equation}
for $t_i$, such that $ap_{n_i-1}\in\left[2^{4^{t_i}},2^{4^{t_i+1}}\right)$ and
\begin{equation}
\label{gamma0nonisome51}
\zeta\left(h(\log_2l_i))\right)=h\left(\log_22q_{m_i-1}^2\right)=2^{-s_i} \log_2 2q_{m_i-1}^2 +2^{s_i+1}-2
\end{equation}
for $s_i$, such that $2q_{m_i-1}^2\in\left[2^{4^{s_i}},2^{4^{s_i+1}}\right)$. Thus, the condition (\ref{gamma0nonisome4}) and the upper estimation of $b$ imply that 
\[\log_2 ap_{n_i-1}^{2r/b(1-\delta)}<2^{t_i-s_i}\log_22q_{m_i-1}^2+2^{t_i+s_i+1}-\frac{2r}{b(1-\delta)}\left(2^{s_i+1}-2^{2s_i+1}\right)\leq \log_22q_{m_i-1}^2\] for sufficiently large $i$. Therefore (from (\ref{gamma0nonisome4})) we have
\begin{equation}
\label{gamma0nonisome6}
ap_{n_i}\geq l_i >2q_{m_i-1}^2>\left(ap_{n_i-1}\right)^{2/b(1-\delta)}>\left(ap_{n_i-1}\right)^8.
\end{equation}

For sufficiently large $i$~we have also, that $\left(ap_{n_i-1}\right)^8>\left(6p_{n_i-1}^2+1\right)^{5/4}$ and
\begin{equation}
\label{gamma0nonisome7}
2q_{m_i-1}^2 \in \left[\left(6p_{n_i-1}^2+1\right)^{5/4},ap_{n_i}\right] 
\end{equation}

Fix $\varepsilon < b\delta/(1/4 -b)$ and define $k_i:=2q_{m_i-1}^2$. Lemma \ref{Lemma1.19} implies that for sufficiently large $i$~we have that $k_i\in\left[(6p_{n_i-1}^2+1)^{5/4},ap_{n_i}\right]\cap\mathbb{N}$, where $6p_{n_i-1}^2\in\left[2^{4^{s_{i,1}}},2^{4^{s_{i,1}+1}}\right)$ and $12p_{n_i-1}^2+k_i\in\left[2^{4^{s_{i,2}}},2^{4^{s_{i,2}+1}}\right)$. Thus, we obtain
\begin{eqnarray}
\liminf_{n\to\infty}\frac{H\left(g,\mcp_n^{[0,1)}\right)}{\zeta\left(h(\log_2n)\right)} &\leq & \liminf_{i\to\infty}\frac{H\left(g,\mcp_{k_i}^{[0,1)}\right)}{h(\log_2k_i)}\nonumber \\
&\leq &\liminf_{i\to\infty}\left(2^{-s_i}\log_2k_i+2^{s_i+1}\right)^{-1}\left[r\left(2^{-s_{i,1}}\log_26p_{n_i-1}^2\right.\right.+\nonumber \\ && \left.+\;2^{s_{i,1}+1}\right) +\left.(k_i/p_{n_i})(1+\varepsilon)\left[2^{-s_{i,2}}\log_2 k_i+2^{-s_{i,2}}+2^{s_{i,2}+1}\right]\right. \nonumber \end{eqnarray}
\begin{eqnarray} &&
  +\;d_{\xi_0}+\max_{x\in[0,1]}g(x)\Large] \nonumber \\
&\leq &\liminf_{i\to\infty}\frac{(k_i/p_{n_i})(1+\varepsilon)\left[2^{-s_{i,2}}\log_2 k_i+2^{-s_{i,2}}+2^{s_{i,2}+1}\right]}{{2^{-s_i}\log_2k_i+2^{s_i+1}}}+\nonumber \\ && +\;
r\cdot \liminf_{i\to\infty}\frac{2^{-s_{i,1}}\log_26p_{n_i-1}^2+2^{s_{i,1}+1}}{{2^{-s_i}\log_2k_i+2^{s_i+1}}}\nonumber
\end{eqnarray}
To estimate the first of the lower limits in the above inequality, we have to consider two cases:\\
{\bf Case 1.} $k_i, k_i+12p_{n_i-1}^2\in \left[2^{4^{s_i}},2^{4^{s_i+1}}\right)$, i.e. $s_{i,2}=s_i$. Then we have
\[\liminf_{i\to\infty}\frac{2^{-s_{i,2}}\log_2 k_i+2^{-s_{i,2}}+2^{s_{i,2}+1}}{{2^{-s_i}\log_2k_i+2^{s_i+1}}}\leq \liminf_{i\to\infty}\frac{ 2^{-s_{i}}\log_2 k_i+ 2^{s_{i}+1}}{{2^{-s_i}\log_2k_i+2^{s_i+1}}}=1.\]
{\bf Case 2.} $k_i\in \left[2^{4^{s_i}},2^{4^{s_i+1}}\right)$ and $k_i+12p_{n_i-1}^2\in\left[2^{4^{s_i+1}},2^{4^{s_i+1}+1}\right)$.
Then by the monotonicity of the~function given by (\ref{fxrosnaca}) in the considered interval, we have%\footnote{Dla $k_i=2^{4^{s_i+1}}$ rozpatrywane wyrażenie zmierza (przy $i$ zmierzającym do plus nieskończoności) do 1, a funkcja $f\colon (0,\infty)\to\mathbb{R}$ dana wzorem $f(x)=(\frac {1}{2a}x+4a)/(\frac 1ax+2a)$ jest malejąca w przedziale $\left(2^{4^{s_i+1}-1},2^{4^{s_i+1}}\right)$.} 
\[\liminf_{i\to\infty}\frac{2^{-s_{i,2}}\log_2 k_i+2^{-s_{i,2}}+2^{s_{i,2}+1}}{{2^{-s_i}\log_2k_i+2^{s_i+1}}}\leq \liminf_{i\to\infty}\frac{\frac 12 2^{-s_{i}}\log_2 k_i+2\cdot 2^{s_{i}+1}}{{2^{-s_i}\log_2k_i+2^{s_i+1}}}\leq 1.\]
Therefore we obtain the following estimation:
\[\liminf_{i\to\infty}\frac{(k_i/p_{n_i})(1+\varepsilon)\left[2^{-s_{i,2}}\log_2 k_i+2^{-s_{i,2}}+2^{s_{i,2}+1}\right]}{{2^{-s_i}\log_2k_i+2^{s_i+1}}}\leq \liminf_{i\to\infty}\frac{k_i}{p_{n_i}}(1+\varepsilon)\leq a(1+\varepsilon).\]
Estimation of the second lower limit is a~consequence of the subadditivity of $h$:
\begin{eqnarray}
\liminf_{i\to\infty}\frac{h(\log_26p_{n_i-1}^2)}{h(\log_2k_i)}&=& \liminf_{i\to\infty}\frac{h(2\log_2ap_{n_i-1}+\log_2(6/a^2))}{h(\log_2k_i)}\nonumber \\
&\leq & \liminf_{i\to\infty}\frac{h(2\log_2ap_{n_i-1})+h(\log_2(6/a^2))}{h(\log_2k_i)}\nonumber\end{eqnarray} \begin{eqnarray}  &=& \liminf_{i\to\infty}\frac{h(2\log_2ap_{n_i-1})}{h(\log_2k_i)}\leq  2\liminf_{i\to\infty}\frac{h(\log_2ap_{n_i-1})}{h(\log_2k_i)}\nonumber \\
&\leq & 2\liminf_{i\to\infty}\frac{a\xi\left(h(\log_2l_i)\right)}{\zeta\left(h(\log_2l_i)\right)}. \nonumber
\end{eqnarray}
Eventually, we obtain
\begin{eqnarray}
\liminf_{n\to\infty}\frac{H\left(g,\mcp_{n-1}^{[0,1)}\right)}{\zeta\left(h(\log_2n)\right)} &\leq & a(1+\varepsilon) +2\beta\liminf_{i\to\infty}\frac{a\xi\left(h(\log_2l_i)\right)}{\zeta\left(h(\log_2l_i)\right)} \nonumber  \\ &\leq & a(1+\varepsilon) + b(1-\delta) \nonumber \\&<& \left(\frac 14 -b\right)\left(1+\frac{b\delta}{1/4 -b}\right)+b(1-\delta)=\frac 14. \nonumber \end{eqnarray}

Therefore $(J_{\xi_0},T_{\xi_0})$ is not of type $\left(\lig(g)\geq \frac 14\right)$ for $\left(\left(\zeta \left(h(\log_2n)\right)\right),P(J_{\xi_0})\right)$.
\end{proof}

It remains to show Lemma \ref{Lemma1.19}. We begin stating the estimation from the proof of \cite[Lemma 4.21]{Blumenotpubl}. 
\begin{fact}\label{lemmakmuxi}
Under the assumptions of Lemma  \ref{Lemma1.19} for every $k\in\left[(6p_{n-1}^2+1)^{5/4},ap_n\right]\cap\mathbb{N}$ we have \[k\mu_{\xi_0}(B)+1<d_{\xi_0}.\]
\end{fact}
\begin{proof}[Proof of Lemma \ref{Lemma1.19}] We will repeat steps of the proof of \cite[Lemma 4.21]{Blumenotpubl} for $g$,~with few (needed) modifications. 
%Dzięki założeniu o ograniczonym przyroście wyrazów ciągu $(p_n)_{n=0}^{\infty}$ uda nam się uzupełnić lukę w dowodzie Blumego.
Fix $\varepsilon\in (0,1)$. Boundedness of $p_{n-1}^3/p_n$ implies that \[\lim_{n\to\infty}\frac{(6p_{n-1}^2+1)^{5/4}}{ap_n}=0.\]
Fom the construction of $\Gamma$ we obtain that $\lim\limits_{n\to\infty}\mu_{\xi}(\tau_n)=1$. Therefore there exists $N>0$, such that for every $n\geq N$ we have
\begin{equation}
\label{lemma119.1} \mu_{\xi_0}(\tau_n)>\frac 12,
\end{equation}
\begin{equation}
\label{lemma119.2} ap_n>(6p_{n-1}^2+1)^{5/4},
\end{equation}
\begin{equation}
\label{lemma119.3} \frac{2}{(6p_{n-1}^2+1)^{1/4}}<\varepsilon,
\end{equation}
and
\begin{equation}
\label{lemma119.4} \frac{6p_{n-1}^2+1}{p_n}<a.
\end{equation}
Let $n\geq N$, $q:=6p_{n-1}^2+1$ and $k\in[q^{5/4},ap_n]\cap\mathbb{N}$. Denote by $\rho_n$ the tower obtained in Step 4. and by $\sigma_n$ tower from Step 5. Then the tower $\rho_n$ is of height $p_n=k_q+j_n$, where
\begin{equation}
\label{lemma119.5} k_n=\left\lfloor \frac{p_n}{q}\right\rfloor \;\;\text{with}\;\; j_n<q,
\end{equation}
and $\sigma_n$ is of height $p_n^2$. We divide  $J_{\xi_0}$ into three sets and calculate $g$-entropy of $\mcp_{k}^{[0,1)}$ with respect to each of these sets separately. Let $l\in\mathbb{N}$ be such that
\begin{equation}
\label{lemma119.6} lq\leq k_nq-k\leq (l+1)q
\end{equation}
Since $k<ap_n<p_n/4$ and $p_n<(k_n+1)q$ we have $k_nq-k>p_n/2$. Thus, $l>1$. Let us define sets
\[A:=(\rho_n)_0^{lq-1},\;\; B:=(\sigma_n)_{p_n^2-p_n+lq}^{p_n^2-1}\cup (J_{\xi_0}\backslash \sigma_n),\]  \[C:=J_{\xi_0}\backslash(A\cup B)=\bigcup_{i=lq}^{p_n-1}\bigcup_{j=0}^{p_n-2}T_{\xi_0}^{i+jp_n}J_n,\] where $J_n$ is a~base of $\sigma_n$. We  define a~partition $\mcq=\{A,B,C\}$. Then, since $g$~is concave, we have
\begin{eqnarray}
H(g,\mcp_k^{[0,1)})&\leq & H(g,\mcp_k^{[0,1)}\vee \mcq) \nonumber\\
&=& H_A(g,\mcp_k^{[0,1)})+H_B(g,\mcp_k^{[0,1)})+H_C(g,\mcp_k^{[0,1)}).\nonumber
\end{eqnarray}

Let us estimate $H_A\left(g,\mcp_k^{[0,1)}\right)$. Denote by $K_n$ the base of the tower $\rho_n$ and by $s([0,1),\rho_n)$ the symbolic representation of an arbitrary point from $K_n$ with respect to $[0,1)$. From the construction of $\rho_n$ we know that first $k_nq$ coordinates of  $s([0,1),\rho_n)$ consists of $k_n$~repetitions of the word of length $q$. Since $k_nq-lq\geq k$ we obtain that on the first  $lq$ levels of $\rho_n$, i.e. in $A$, we will find no more than  $q$~different words of length $k$. Thus, subwords of length $k$ of $s([0,1),\rho_n)$ satisfy the equation
\[s([0,1),\rho_n)_i^{i+k-1}=s([0,1),\rho_n)_{i=jq}^{i+jq+k-1}\]
for $i\in\{0,\ldots,q-1\}$ and $j\in\{1,\ldots,l-1\}$.Therefore \[A\cap\mcp_k^{[0,1)}\preccurlyeq A\cap\left\{\bigcup_{j=0}^{l-1}T_{\xi_0}^{i+jq}K_n\right\}_{i=0}^{q-1}.\]  Hence, from the monotonicity of $h$, equality $\mu_{\xi_0}(\rho_n)=\mu_{\xi_0}(\tau_n)$ and the condition (\ref{lemma119.1}) we obtain
\begin{eqnarray}
H_A\left(g,\mcp_k^{[0,1)}\right)&\leq & H_A\left(g,\left\{\bigcup_{j=0}^{l-1}T_{\xi_0}^{i+jq}K_n\right\}_{i=0}^{q-1}\right) = q\cdot g\left(\frac{\mu_{\xi_0}(A)}{q}\right)=q\cdot g(l\mu_{\xi_0}(K_n))\nonumber \\ &=& q\cdot g\left(\frac{l\mu_{\xi_0}(\rho_n)}{p_n}\right)=\frac{lq\mu_{\xi_0}(\rho_n)}{p_n}\cdot  h\left(-\log_2\frac{l\mu_{\xi_0}(\rho_n)}{p_n}\right) \nonumber\end{eqnarray}  \begin{eqnarray}&<& h\left(\log_2\frac{p_n}{l\mu_{\xi_0}(\rho_n)}\right)
<  h(\log_22p_n)\nonumber
\end{eqnarray}
The assumption (\ref{pnr}) and subadditivity of $h$~imply
\begin{eqnarray} h(\log_22p_n)&\leq & h\left(\log_22p_{n-1}^{2r}\right)=h(1+r\log_2p_{n-1}^2)\nonumber \\ &\leq & h(1)+r h(\log_2p_{n-1}^2)\leq 1+r h(\log_26p_{n-1}^2). \nonumber \end{eqnarray}
Therefore \begin{equation}\label{HA} H_A\left(g,\mcp_k^{[0,1)}\right)\leq1+r\left(2^{-m_{n,1}}\log_26p_{n-1}^2+2^{m_{n,1}+1}\right)\end{equation}
where $m_{n,1}$ is such that $6p_{n-1}^2\in\left[2^{4^{m_{n,1}}},2^{4^{m_{n,1}+1}}\right)$.

 Let's estimate  $H_B(g,\mcp_k^{[0,1)})$. We know that $\mu_{\xi_0}(B)$ is small since $B$ is an algebraic sum of the complement of the tower $\sigma_n$ and few highest levels of the tower $\sigma_n$. 
Therefore application of Lemma \ref{gxiogr} and Corollary \ref{lemmakmuxi} gives us the following  estimation
\begin{equation} \label{HB}  H_B\left(g,\mcp_k^{[0,1)}\right)=\sum_{s\in\{0,1\}^k}g\left(\mu_{\xi_0}\left(A_k^{[0,1)}(s)\cap B\right)\right)<k\mu_{\xi_0}(B)+1  \leq d_{\xi}-8, \end{equation}
where $A_k^{[0,1)}(s)$ is given by (\ref{AnEs}).

It remains to estimate $H_C(g,\mcp_k^{[0,1)})$. Let us denote by $s([0,1),\sigma_n)$ a~symbolic representation of a~point from $J_n$ with respect to  $[0,1)$. It follows from the definition of  $\sigma_n$, that $s([0,1),\sigma_n)$ consists of  $p_n$ repetitions of $s([0,1),\rho_n)$.  Thus, subwords of length $k$ of $s([0,1),\sigma_n)$ fulfill the equation
\[s([0,1),\sigma_n)_i^{i+k-1}=s([0,1),\rho_n)_{i=jq}^{i+jq+k-1}\]
for $i\in\{0,\ldots,p_n-1\}$ and $j\in\{1,\ldots,p_n-2\}$. Therefore
 \[C\cap\mcp_k^{[0,1)}\preccurlyeq C\cap\left\{\bigcup\limits_{j=0}^{p_n-2}T_{\xi}^{i+jq}J_n\right\}_{i=0}^{p_n-q}.\] Using this fact, monotonicity of $h$ and properties (\ref{lemma119.3}), (\ref{lemma119.5}) and (\ref{lemma119.6}), we obtain
\begin{eqnarray}
H_C\left(g,\mcp_k^{[0,1)}\right) &\leq & H_C\left(g,\left\{\bigcup_{j=0}^{p_n-2}T_{\xi_0}^{i+jq}J_n\right\}_{i=0}^{p_n-q}\right)\nonumber
\leq  (p_n-q)g\left(\frac{\mu_{\xi_0}(C)}{p_n-q}\right)\\ &\leq & g(\mu_{\xi_0}(C))+\mu_{\xi}(C)h\left(\log_2(p_n-q)\right)\nonumber \\
&\leq & (p_n-lq)(p_n-1)\mu_{\xi_0}(J_n)h(\log_2(p_n-q)) +\max_{x\in[0,1]}g(x) \nonumber 
\end{eqnarray}
\begin{eqnarray}
&\leq & \frac{p_n-lq}{p_n}h(\log_2(p_n-q))+1\nonumber \\&<& \frac{p_n-(k_n-1)q+k}{p_n}h(\log_2(p_n-(k_n-1)q+k))+1\nonumber \\ &<& \frac{2q+k}{p_n}h(\log_2(2q+k))+1 \nonumber \\
&= & \frac{k}{p_n}\left(\frac{2q}{k}+1\right)h(\log_2(2q+k))+1 \nonumber \\
&\leq & \frac{k}{p_n}\left(\frac{2q}{q^{5/4}}+1\right)h(\log_2(2q+k))+1 \nonumber \\
&< & \frac{k}{p_n}\left(1+\varepsilon\right)h(\log_2(2q+k))+1 \nonumber 
 \end{eqnarray}
Therefore
\begin{eqnarray}
H_C\left(g,\mcp_k^{[0,1)}\right)&<& \frac{k}{p_n}\left(1+\varepsilon\right)\left[2^{-m_{n,2}}\left(\log_2k+\log_2\left(\frac{2q}{k}+1\right)\right)+2^{m_{n,2}+1}-2\right]+1 \nonumber \\
&< & \frac{k}{p_n}\left(1+\varepsilon\right)\left[2^{-m_{n,2}}\left(\log_2k+\log_2\left(1+\varepsilon\right)\right)+2^{m_{n,2}+1}\right]+1 \nonumber
\end{eqnarray}
where $m_{n,2}$ is such that $2q+k\in[2^{4^{m_{n,2}}},2^{4^{m_{n,2}+1}})$.
Hence \begin{equation} H_C\left(g,\mcp_k^{[0,1)}\right)< \frac{k}{p_n}\left(1+\varepsilon\right)\left[2^{-m_{n,2}}\left(\log_2k+1\right)+2^{m_{n,2}+1}\right]+1, \label{HC}
\end{equation}

Putting together (\ref{HA}), (\ref{HB}) and (\ref{HC}) we obtain the assertion.
\end{proof}

It is worth noting that the crucial property in this note was subadditivity of $h$. It implies subderivativity of $g$ and hence subadditivity of $(H(g,\mcp_n))$ with respect to the partition $\mcp_n$. It can be expected that (modulo some necessary estimates) similar results can be obtained for other subderivative functions. It is possible for example, for $g_0$ (from the previous section). But in this case we would obtain just a~special case of \cite[Thm 4.22]{Blumenotpubl}, since then the condition (\ref{gamma0nonisome1}) for $h(x)=\log_2(1+x)$ is \[\liminf_{n\to\infty} \frac{a\xi_0 \left(\log_2(1+\log_2n) \right)}{\zeta \left(\log_2(1+\log_2n)\right)}<\frac{b}{2r}\] and it implies (\ref{bbbbbbbbbbb}). The choice of not regularly varying function $h$ comes from the fact, that there should exist systems $(J_{\xi},T_{\xi}), (J_{\xi_0},T_{\xi_0})\in\Gamma$ such that the condition (\ref{bbbbbbbbbbb}) doesn't hold, while (\ref{gamma0nonisome1}) does.

%\section*{References}

\end{document}